\documentclass[12pt]{amsart}
\usepackage{amsmath}
\usepackage{amsfonts,amssymb}
\usepackage{amsthm}
\usepackage[matrix,arrow,curve]{xy}
\numberwithin{equation}{section}
\theoremstyle{plain}
\newtheorem{theorem}{Theorem}[section]
\newtheorem{lemma}[theorem]{Lemma}
\newtheorem{predl}[theorem]{Proposition}
\newtheorem{corollary}[theorem]{Corollary}

\theoremstyle{definition}
\newtheorem{definition}[theorem]{Definition}
\newtheorem{remark}[theorem]{Remark}
\newtheorem{example}[theorem]{Example}
\newcommand{\Q}{\mathbb Q}

\renewcommand{\P}{\mathbb P}
\renewcommand{\AA}{\mathcal A}

\newcommand{\D}{\mathcal D}

\newcommand{\EE}{\mathcal E}
\newcommand{\FF}{\mathcal F}

\newcommand{\II}{\mathcal I}

\newcommand{\NN}{\mathcal N}

\renewcommand{\O}{\mathcal O}

\renewcommand{\k}{\mathsf k}

\newcommand{\modd}{\mathrm{mod{-}}}

\newcommand{\coh}{\mathrm{coh}}

\renewcommand{\le}{\leqslant}
\renewcommand{\ge}{\geqslant}

\def\a{\alpha}
\def\b{\beta}

\DeclareMathOperator{\Hom}{\textup{Hom}}

\DeclareMathOperator{\Ext}{\textup{Ext}}

\DeclareMathOperator{\coker}{\mathrm{coker}}

\usepackage{array}
\begin{document}

\title[A criterion for left-orthogonality on a surface]{A criterion for left-orthogonality of an effective divisor on a surface}

\author{Alexey ELAGIN}
\thanks{The author was partially supported by the Russian Academic Excellence Project '5-100', by Simons-IUM fellowship and by RFBR projects 15-01-02158 and 15-51-50045.
}
\address{Institute for Information Transmission Problems (Kharkevich Institute), Moscow, RUSSIA;\\
National Research University Higher School of Economics, Moscow, RUSSIA}
\email{alexelagin@rambler.ru}

\begin{abstract}We find a criterion for an effective divisor $D$ on a smooth surface to be left-orthogonal or strongly left-orthogonal (i.e. for the pair of line bundles $(\O,\O(D))$ to be exceptional or strong exceptional).
\end{abstract}

\maketitle

%\tableofcontents

\section{Introduction}

Having a full exceptional collection in the derived category of coherent sheaves is a nice but rare property of an algebraic  variety. 
Starting from the 1980-s, a series of examples of full exceptional collections on different varieties was constructed.
Among such varieties are projective spaces, Grassmann varieties and quadrics, some other homogeneous spaces,
del Pezzo surfaces, toric Fano 3-folds, some other Fano 3-folds. All these collections consist of vector bundles. 

On the other hand, for most varieties it is easy to demonstrate that full exceptional collections do not exist. For example, they do not exist if $K_0(\coh(X))$ is not a lattice. Nevertheless, the following folklore conjecture seems to be out of reach: any variety with a full exceptional collection in the derived category is rational.

Among the full exceptional collections, strong ones are the nicest. 
By a theorem of A.\,Bondal \cite{Bo}, for a variety $X$ admitting a full strong exceptional collection $(\EE_1,\ldots,\EE_n)$ 	
there is an equivalence of categories 
\begin{equation*}
\D^b(\coh (X))\cong \D^b(\modd \AA)
\end{equation*}
where $\AA$ is the endomorphism algebra of the object $\oplus_i\EE_i$ and $\modd \AA$ is the category of right finitely generated $\AA$-modules. 
Full strong exceptional collections exist on projective spaces, blow-ups of a projective plane in several points (in particular, on del Pezzo surfaces), quadrics, Grassmann varieties, toric Fano 3-folds.

It was conjectured by A.\,King (see \cite{K1}) that every smooth toric variety has a full strong exceptional collection of line bundles. In \cite{HP1} L.\,Hille and M.\,Perling described a smooth toric surface which does not have such a collection (hence producing a counter example). The paper \cite{HP2} by Hille and Perling contains a first systematic study of 
full exceptional collections of line bundles on surfaces. They proved, in particular, that a toric surface has a full strong exceptional collection of line bundles if and only if it can be obtained from a Hirzebruch surface by two blow-ups (several points can be blown-up on each step).
M.\,Brown and I.\,Shipman proved in \cite{BS} that any surface admitting a full strong exceptional collection of line bundles is rational.
Still it is not known which rational surfaces possess such collections.

Following ideas of \cite{HP2}, having a collection
$$(\O(D_1),\ldots,\O(D_n))$$ of line bundles one should consider differences $D_j-D_i$ between divisors in this collection. 
Clearly, semiorthogonality of line bundles $\O(D_i)$ and $\O(D_j)$ is equivalent to vanishing of cohomology of the divisor $D_i-D_j$. This motivates the notions of left-orthogonal and strongly left-orthogonal divisors, see Definition~\ref{def_lo} below. Briefly, a divisor $D$ on a smooth rational surface $X$ is \emph{left-orthogonal} if the pair $(\O_X,\O_X(D))$ is an exceptional pair. Divisor $D$ is \emph{strongly left-orthogonal} if the pair $(\O_X,\O_X(D))$ is a strong exceptional pair.

In this note we express left-orthogonality and strong left-orthogonality of an effective divisor on a surface via geometry and combinatorics of its components. Main results are the following criteria (see Theorems~\ref{th_slo} and \ref{th_slo2}):

\begin{theorem}
An effective divisor $D$ on a surface $X$ with $h^1(X,\O_X)=h^2(X,\O_X)=0$ is left-orthogonal if and only if the following conditions hold
\begin{enumerate}
\item $D$ is a tree of projective lines.
\item
$p_a(D)=0$ and for any connected divisor $D'$ such that $0<D'\le D$ one has $p_a(D')\le 0$.
\end{enumerate}
\end{theorem}

\begin{theorem}
An effective divisor $D$ on a surface $X$ with $h^1(X,\O_X)=h^2(X,\O_X)=0$ is strongly left-orthogonal if and only if the following conditions hold
\begin{enumerate}
\item $D$ is a tree of projective lines.
\item
$p_a(D)=0$ and for any connected divisor $D'$ such that $0<D'\le D$ one has $p_a(D')\le 0$.
\item
For any connected divisor $D'$ such that $0<D'\le D$ one has $p_a(D')\le 1+D\cdot D'$.
\end{enumerate}
\end{theorem}

Given characterization of left-orthogonal divisors was used in the first version of the  recent paper \cite{EL} of V.\,Lunts and the author in order to prove that any full strong exceptional collection of line bundles on a del Pezzo surface can be obtained by a certain explicit construction called \emph{standard augmentation}. We refer to \cite{EL} or \cite{HP2} for details. In the further versions we have got rid of studying geometry of left-orthogonal divisors and using the strongness condition.
Still we expect that the criteria of left-orthogonality from this note can be useful for the investigation on the non-del Pezzo case.

This work has grown from the collaboration with Valery Lunts to whom I am kindly grateful.

\section{Left-orthogonality of divisors on a surface}
\label{section_lo}
%\subsubsection{Hodge index theorem}
%Let $X$ be a surface and $\rho$ be the rank of the Neron-Severi group of $X$. Then the %intersection form on $NS(X)$ has signature $(1,\rho-1)$.

By a surface, in this note we mean a smooth projective surface over an algebraically closed field $\k$ of zero characteristic.

In this note we are interested in exceptional pairs of line bundles on surfaces. 
We recall that  a sheaf $\FF$ on an algebraic variety is called \emph{exceptional} if $\Hom(\FF,\FF)=\k$  and $\Ext^i(\FF,\FF)=0$ for $i>0$. A pair $(\FF_1,\FF_2)$ of sheaves is called \emph{exceptional} if $\FF_1$ and $\FF_2$ are exceptional and $\Ext^i(\FF_2,\FF_1)=0$ for all $i$. A pair $(\FF_1,\FF_2)$ of sheaves is called \emph{strong exceptional} if in addition  $\Ext^i(\FF_1,\FF_2)=0$ for $i>0$.

Clearly, a line bundle $\EE$ on a smooth projective surface $X$ is exceptional if and only if the structure sheaf $\O_X$ is exceptional which is equivalent to $h^1(X,\O_X)=h^2(X,\O_X)=0$. We note here that these conditions are satisfied for any rational surface. On other hand, there are irrational surfaces whose structure sheaf is exceptional: for example, Enriquez surfaces.   

Consider a pair of line bundles $(\O_X(D_1),\O_X(D_2))$ on a surface $X$. Clearly, it is an exceptional pair if and only if $\O_X$ is an exceptional sheaf and
$$h^i(X,\O(D_1-D_2))=0\qquad\text{for}\qquad i=0,1,2.$$
This motivates the following definition given in \cite{HP2}:

\begin{definition}
\label{def_lo}
We say that a divisor $D$ on a surface $X$ is \emph{left-orthogonal} if $$H^i(X,\O_X(-D))=0$$ for all $i$.

We say that a divisor $D$ on a surface $X$ is \emph{strongly left-orthogonal} if $D$ is left-orthogonal and  $$H^i(X,\O_X(D))=0$$ for $i=1,2$.
\end{definition}

The below proposition immediately follows from definitions.
\begin{predl}
A collection of line bundles
$$(\O_X(D_1),\ldots,\O_X(D_n))$$
on a surface is (strong and) exceptional if and only if the sheaf $\O_X$ is exceptional and
for any $1\le i<j\le n$ the divisor $D_j-D_i$ is (strongly) left-orthogonal. 
\end{predl}

\section{Some preliminaries}

\subsection{Arithmetic genus}
Let $Z$ be a scheme. Its \emph{arithmetic genus} is defined as 
$$p_a(Z)=1-\chi(Z,\O_Z)=1-\sum_i(-1)^ih^i(Z,\O_Z).$$
For an effective divisor $D$ on a surface $X$, one can consider $D$ as a (maybe non-reduced) subscheme of $X$. Hence one can speak about arithmetic genus of $D$. It can be calculated using Riemann-Roch formula. %, its arithmetic genus can be expressed as
Indeed,
$$\chi(\O_D)=\chi(\O_X)-\chi(\O_X(-D))=\chi(\O_X)-(\chi(\O_X)+\frac12(-D(-D-K_X)))=
-\frac12D(D+K_X)$$
and
$$p_a(D)=1+\frac12D(D+K_X).$$

\subsection{Cohomology vanishing for non-reduced schemes}
%We recall some facts about structure sheaves of non-reduced schemes. 
Suppose $E\subset X$ is a smooth curve on a surface, $E\cong \P^1$. Denote by $\II=\II_E\subset \O_X$ the sheaf of ideals of~$E$. Then $\II/\II^2\cong \NN_{X,E}^*\cong \O_E(-r)$ where $r=E\cdot E$. Let $D=kE$ be the non-reduced closed subscheme of $X$ defined by the ideal sheaf $\II^k$. Then the structure sheaf $\O_D$ has a quotient-filtration
$$\O_{D}=\O_{kE}\to \O_{(k-1)E}\to\ldots\to \O_{2E}\to \O_E\to 0$$
with kernels
$$\ker(\O_{(i+1)E}\to \O_{iE})\cong \II^{i}/\II^{i+1}\cong (\II/\II^{2})^{\otimes i}\cong
 \O_E(-ir),\qquad{i=0,1,\ldots,k-1}.$$
Let $b$ be an integer.

\begin{lemma}
\label{lemma_cond3}
In the above notation  $H^1(X,\O_{kE}(b))=0$ if one of the following conditions hold: 
\begin{enumerate}
\item $b\ge -1, r\le 0$,
\item $b\ge -1, k=1$,
\item $b\ge 0, r=1, k=2$.
\end{enumerate}
\end{lemma}

\begin{proof}
The proof is by induction in $k$. Note that it is reasonable since the above conditions are stable under decreasing of $k$.

For $k=1$, we have $H^1(X,\O_E(b))=H^1(\P^1,\O_{\P^1}(b))=0$ since $b\ge -1$.

For the induction step, suppose that $k\ge 2$. 
Consider exact sequence
$$0\to \II^{k-1}/\II^{k}(b)\to \O_{kE}(b)\to \O_{(k-1)E}(b)\to 0,$$
it gives an exact sequence
$$H^1(X,\II^{k-1}/\II^{k}(b))\to H^1(X,\O_{kE}(b))\to H^1(X,\O_{(k-1)E}(b)).$$
By induction hypothesis, $H^1(X,\O_{(k-1)E}(b))=0$. Also, 
$\II^{k-1}/\II^{k}(b)\cong \O_{\P^1}(b-r(k-1))$ where $b-r(k-1)\ge -1$. Therefore
$H^1(X,\II^{k-1}/\II^{k}(b))=0$ and
$H^1(X,\O_{kE}(b))=0$ what concludes the proof.
\end{proof}

\subsection{Intersection with canonical divisor}
The next easy fact follows directly from adjunction formula.
\begin{lemma}
Let $E\subset X$ be a smooth curve on a surface, $E\cong \P^1$. Then
$$E\cdot K_X=-2-E\cdot E.$$
\end{lemma}

\section{Criterion for left-orthogonality}	
Further we make the following assumptions. Suppose $X$ is a surface with $h^1(X,\O_X)=h^2(X,\O_X)=0$ (for instance, a rational surface). Let $D$ be an effective divisor on $X$. Suppose $D=\sum k_iE_i$  where $E_i$ are prime divisors, let $r_i=E_i^2$. Denote $D_i=k_iE_i$. 
Also we denote by $D_i$ and $D$ the corresponding (possibly) non-reduced subschemes of~$X$.

We determine using these data whether $D$ is left-orthogonal and strongly left-orthogonal.

We start with two simple lemmas.
\begin{lemma}
\label{lemma_lo}
Let $D$ be an effective divisor on a surface $X$ with $h^1(X,\O_X)=h^2(X,\O_X)=0$. Then:
\begin{enumerate}
\item
Divisor $D$ is left-orthogonal if and only if for the structure sheaf of a closed subscheme $D\subset X$ one has
\begin{equation}
\label{eq_h0h1h2}
h^0(X,\O_D)=1,\qquad h^1(X,\O_D)=0.
\end{equation}
\item Divisor $D$ is strongly left-orthogonal if and only if $D$ is left-orthogonal  and one has 
\begin{equation}
%\label{eq_h0h1h2}
h^1(X,\O_D(D))=0.
\end{equation}
\end{enumerate}
\end{lemma}
\begin{proof}
\begin{enumerate}
\item 
Consider the exact sequence of sheaves on~$X$
$$0\to \O_X(-D)\to \O_X\to \O_D\to 0.$$
Its long exact sequence of cohomology implies that $H^i(X,\O_X(-D))=0$ for all~$i$ $\iff H^i(X,\O_D)\cong H^i(X,\O_X)$.
%Since $X$ is rational, one has $h^0(X,\O_X)=1$, $h^1(X,\O_X)=h^2(X,\O_X)=0$.
Keep in mind that $h^2(X,\O_D)=0$ because $D$ is a one-dimensional scheme.
\item  Consider the exact sequence of sheaves 
on~$X$:
\begin{equation}
\label{eq_OD}
0\to \O_X\to\O_X(D)\to\O_D(D)\to 0.
\end{equation}
Note that $H^2(X,\O_D(D))=0$ because the sheaf $\O_D(D)$ is supported in dimension~$1$. Therefore the long exact sequence of cohomology implies that a left-orthogonal divisor $D$ is strongly left-orthogonal if and only if $H^1(X,\O_D(D))=0$.
\end{enumerate}
\end{proof}

\begin{lemma}
\label{lemma_subdivisor}
Suppose $D$ is an effective left-orthogonal divisor on a surface $X$ with $h^1(X,\O_X)=h^2(X,\O_X)=0$. Let $D'\le D$ be an effective reduced connected divisor. Then $D'$ is also left-orthogonal. 
\end{lemma}
\begin{proof}
We treat $D$ and $D'$ as closed subschemes of $X$. Note that $D'$ is a closed subscheme of $D$. Consider the exact sequence of sheaves on $X$:
$$0\to \II_{D,D'}\to \O_D\to\O_{D'}\to 0,$$
where $\II_{D,D'}$ denotes the sheaf of ideals of the closed subscheme $D'\subset D$.
Its long exact sequence of cohomology has a fragment
$$H^1(X,\O_D)\to H^1(X,\O_{D'})\to H^2(X,\II_{D,D'}).$$
Since the sheaf of ideals $\II_{D,D'}$ is supported on $D$, its second cohomology vanishes.
By Lemma \ref{lemma_lo}.(1) we have $H^1(X,\O_D)=0$, it follows that $H^1(X,\O_{D'})=0$. Also, we have $H^0(X,\O_{D'})=\k$ because $D'$ is connected and reduced. By Lemma \ref{lemma_lo}.(1), the divisor $D'$ is left-orthogonal.
\end{proof}

Suppose that effective divisors $C_1,\ldots,C_n$ on a smooth surface $X$ have no common components, denote $C=\sum C_i$. We treat $C$ and $C_i$ as (maybe non-reduced) subschemes of $X$. In this section we will consider the following 
sequence of sheaves on $X$:
\begin{equation}
\label{eq_cutting}
0\to \O_C\to \bigoplus_i \O_{C_i}\to \bigoplus_{i<j}\bigoplus_{P\in C_i\cap C_j}\O_{(C_i\cap C_j)_P}\to 0,
\end{equation}
where $(C_i\cap C_j)_P$ denotes the scheme-theoretic intersection supported in the point $P$.
Here the map $\O_C\to \oplus_i \O_{C_i}$ is the sum of restrictions and the map 
$\oplus_i\O_{C_i}\to \oplus_{i<j}\oplus_{P\in C_i\cap C_j}\O_{(C_i\cap C_j)_P}$ is a collection of differences.
This sequence is exact if at any point $P$ of intersection of $C_i$-s only two divisors  meet.

\begin{predl}
\label{prop_tree}
Let $D=\sum k_iE_i$ be an effective left-orthogonal divisor on a surface $X$ with $h^1(X,\O_X)=h^2(X,\O_X)=0$. Then every component $E_i$ is isomorphic to $\P^1$, any intersection of components $E_i\cap E_j$ is transversal and components of $D$ form a tree.
\end{predl}

\begin{proof}
Let $E$ be a component of $D$. Then $E$ is an irreducible left-orthogonal divisor by Lemma~\ref{lemma_subdivisor}, we also treat $E$ as a reduced subscheme of $X$. We have 	$H^1(E,\O_E)=H^1(X,\O_E)=0$ by Lemma~\ref{lemma_lo}.(1).
It follows that $E\cong \P^1$, see 
\cite[Exercise IV.1.8b]{Ha}.
%Let $\nu\colon \~E\to E$ be the normalization of $E$. Consider the exact sequence of %sheaves on $E$:
%$$0\to \O_E\to \nu_*\O_{\~E}\to \FF\to 0.$$
%Here $\FF$ is the cokernel of the injection $\O_E\to \nu_*\O_{\~E}$, it is a sheaf with %zero-dimensional support. The associated long exact sequence of cohomology is
%$$0\to H^0(X,\O_E)\to H^0(X,\nu_*\O_{\~E})\to H^0(X,\FF)\to H^1(X,\O_E)\to %H^1(X,\nu_*\O_{\~E})\to H^1(X,\FF).$$
%Since $H^1(X,\O_E)=H^1(X,\FF)=0$, we get 
%$$H^1(\~E,\O_{\~E})\cong H^1(X,\nu_*\O_{\~E})=0,$$
% therefore $\~E\cong \P^1$. 
%Since $H^0(X,\O_E)\to H^0(X,\nu_*\O_{\~E})$ is an isomorphism and $H^1(X,\O_E)=0$, we %obtain $H^0(X,\FF)=0$. It follows that $\FF=0$ because $\FF$ is supported in finite %number of points. Consequently, $\O_E\to\nu_*\O_{\~E}$ is an isomorphism, $\nu$ is an %isomorphism and $E\cong \P^1$ is a smooth curve.

Now consider any pair  of components in $D$ that intersect nontrivially. Let them be $E_1,E_2$ and $P_1,\ldots,P_m$ be the common points of $E_1$ and $E_2$. Take  $C=E_1+E_2$, then $C$ is left-orthogonal by Lemma~\ref{lemma_subdivisor}.
Consider exact sequence (\ref{eq_cutting}) for $C_1=E_1$ and $C_2=E_2$, it has the form:
$$0\to \O_{C}\to \O_{E_1}\oplus \O_{E_2}\to \oplus_{j=1}^m \O_{(E_1\cap E_2)_{P_j}}\to  0.$$
We have $h^1(X,\O_C)=0$,
therefore the long exact sequence of cohomology implies that $\sum_j h^0(X,\O_{(E_1\cap E_2)_{P_j}})=1$. 
Recall that $(E_1\cap E_2)_{P_j}$ denote the scheme-theoretic intersection.
Therefore $m=1$ and $E_1$ and $E_2$ intersect transversally in the unique point $P_1$.

Let us show that no three components of $D$ meet at one point. Suppose $P\in E_1\cap E_2\cap E_3$. Let $E_{23}=E_2+E_3$ and $C=E_1+E_{23}$, then $C,E_1$ and $E_{23}$ are left-orthogonal by Lemma~\ref{lemma_subdivisor}.
Let $P=P_1,\ldots,P_m$ be the common points of $E_1$ and $E_{23}$.
Consider the sequence (\ref{eq_cutting}) for~$C_1=E_1$ and $C_2=E_{23}$. It is exact and has the form:
$$0\to \O_{C}\to \O_{E_1}\oplus \O_{E_{23}}\to \oplus_{j=1}^m \O_{(E_1\cap E_{23})_{P_j}}\to 0.$$
As above, long exact sequence of cohomology implies that $\sum_j h^0(X,\O_{(E_1\cap E_{23})_{P_j}})=1$. But the intersection $(E_1\cap E_{23})_P$ is non-reduced and $h^0(X,\O_{(E_1\cap E_{23})_P})\ge 2$, we get a contradiction.

In the same manner one can demonstrate that components of $D$ cannot form a cycle.
Finally, $D$ is connected because $h^0(X,\O_D)=1$ by Lemma~\ref{lemma_lo}.(1). 
\end{proof}

\begin{definition}
Let $D=\sum k_iE_i$ be an effective divisor on a surface. By saying that $D$ is a \emph{tree of projective lines} we mean that every component $E_i$ of $D$ is isomorphic to $\P^1$, any intersection of components $E_i\cap E_j$ is transversal and components of $D$ form a tree.
\end{definition}

\begin{theorem}
\label{th_slo}
An effective divisor $D$ on a surface $X$ with $h^1(X,\O_X)=h^2(X,\O_X)=0$ is left-orthogonal if and only if the following conditions hold
\begin{enumerate}
\item $D$ is a tree of projective lines.
\item
$p_a(D)=0$ and for any connected divisor $D'$ such that $0<D'\le D$ one has $p_a(D')\le 0$.
\end{enumerate}
\end{theorem}

\begin{proof}
By Proposition~\ref{prop_tree} condition (1) is fulfilled for a left-orthogonal divisor $D$.  

By Lemma~\ref{lemma_lo}, left-orthogonality of $D$ is equivalent to equalities 
$\chi(D)=1, h^1(D)=0$. Suppose these equalities are fulfilled. Then for any effective connected divisor $D'\le D$ one has an exact sequence 
$$0\to \II_{D',D}\to \O_{D}\to \O_{D'}\to 0.$$
It gives an exact sequence 
$$0=H^1(X,\O_{D})\to H^1(X,\O_{D'})\to H^2(X,\II_{D',D})=0$$
therefore $h^1(X,\O_{D'})=0$ and $p_a(D')=1-h^0(D')\le 0$.

The proof in the other direction follows readily from the next lemma.
\end{proof}

\begin{lemma}
\label{lemma_main1}
Let $D$ be an effective divisor  on a surface $X$ with $h^1(X,\O_X)=h^2(X,\O_X)=0$. Suppose $D$ is a tree of projective lines. Suppose also that
$\chi(D')\ge 1$ for any connected divisor $D'$ such that $0<D'\le D$. Then $h^1(D)=0$.
\end{lemma}

\begin{proof}
First we consider the special case $D=kE$ where $E$ is a prime divisor. 
We have
$$1\le \chi(\O_D)=-\frac12 kE(kE+K_X)=\frac12(-k^2r+k(2+r))=\frac k2(2+r(1-k)).$$
Therefore $2+r(1-k)$ is a positive integer and $r(k-1)\le 1$. It follows that for $b=0$ one of conditions from Lemma~\ref{lemma_cond3} is fulfilled and thus $h^1(\O_D)=0$.

The general case $D=\sum_i k_iE_i$ where $E_i$ are prime divisors is treated by induction in $\sum k_i$. The base of induction is contained in the above special case. For the step of induction, consider the  sequence  
(\ref{eq_cutting}) for~$C_i=D_i$:

$$0\to \O_D\to \bigoplus_i \O_{D_i}\to \bigoplus_{i< j}\O_{D_i\cap D_j}\to 0.$$
Since all intersections of components are simple, this sequence is exact.
Its long exact sequence of cohomology starts with 
$$
\xymatrix{
0\ar[r] & H^0(X,\O_D)\ar[rr] && \oplus H^0(X,\O_{D_i})\ar[rr]&& \oplus H^0(X,\O_{D_i\cap D_j})\ar[lllld] &\\
& H^1(X,\O_D)\ar[rr] && \oplus H^1(X,\O_{D_i}). &&& 
}
$$
%Clearly $H^{1}(X,\O_{D_i\cap D_j})=0$ because the sheaf $\O_{D_i\cap D_j}$ is supported %in dimension $0$. 
Note that $H^1(X,\O_{D_i})=0$ by the special case of lemma for any $i$.
Therefore condition  $h^1(X,\O_D)=0$ is equivalent to 
the map
$$\a_D\colon  \oplus H^0(X,\O_{D_i})\to \oplus H^0(X,\O_{D_i\cap D_j})$$
being epimorphic. 

Let $E_a$ be a component of $D$, let $D'=D-E_a$. Clearly, $D'=\sum k'_iE_i$ where $k'_a=k_a-1$ and $k'_i=k_i$ for $i\ne a$. Denote by $\II$ the sheaf of ideals of the subscheme $E_a\subset X$. Then the sheaf of ideals of subscheme $D'_a\subset D_a$ is isomorphic to $\II^{k_a-1}/\II^{k_a}\cong (\II/\II^2)^{k_a-1}\cong \O_{E_a}(-r_a(k_a-1))$. Consider a commutative diagram
\begin{equation}
\tiny
\label{eq_aab}
\xymatrix{
0\ar[r] & H^0(X,\II^{k_a-1}/\II^{k_a}) \ar[r] \ar[d]^{\beta} & 
\oplus H^0(X,\O_{D_i}) \ar[r] \ar[d]^{\a_D} & 
\oplus H^0(X,\O_{D'_i}) \ar[r]\ar[d]^{\a_{D'}} & 0\\
0\ar[r] & \bigoplus\limits_{j: E_j\cap E_a\ne \emptyset, j\ne a}  H^0\left(\left(\II^{k_a-1}/\II^{k_a}\right)|_{D_j}\right) \ar[r] &
\oplus H^0(X,\O_{D_i\cap D_j}) \ar[r] & 
\oplus H^0(X,\O_{D'_i\cap D'_j}) \ar[r] & 0. 
}
\end{equation}
Consider exact sequence of sheaves
\begin{equation}
\label{eq_117}
0\to \II^{k_a-1}/\II^{k_a}\to \O_{k_aE_a}\to \O_{(k_a-1)E_a}\to 0.
\end{equation}
We claim that in its long sequence of cohomology one has $H^1(X,\II^{k_a-1}/\II^{k_a})=0$.
Indeed, $\II^{k_a-1}/\II^{k_a}\cong \O_{E_a}(-r_a(k_a-1))$ and $-r_a(k_a-1)\ge -1$, see the proof of the special case of the lemma. Therefore the first row of (\ref{eq_aab}) is exact. 
Since intersection $E_a\cap E_j$ is transversal by condition (1), restricting sequence (\ref{eq_117}) on $D_j$ ($=D'_j$) we get an exact sequence
$$0\to (\II^{k_a-1}/\II^{k_a})|_{D_j}\to \O_{D_a\cap D_j}\to \O_{D'_a\cap D'_j}\to 0$$
of sheaves supported in dimension $0$. Its $H^0$ make up an exact sequence, such sequences form the second row of~(\ref{eq_aab}). Therefore the second row of (\ref{eq_aab}) is also exact. 

Thus, diagram (\ref{eq_aab}) gives an exact sequence
$$\coker \b\to \coker \a_{D}\to \coker \a_{D'}.$$
We claim that $\a_{D'}$ is epimorphic. If $D'$ is connected, this is so by the induction hypothesis. If $D'=D'_{(1)}+\ldots+D'_{(m)}$ is a sum of connected components then $\a_{D'}$ is a sum of the maps $\a_{D'_{(i)}}$. Each of them is epimorphic by the induction hypothesis, hence $\a_{D'}$ is also epimorphic.
Consequently, if $\b$ is epimorphic then $\a_{D}$ also is epimorphic. Hence, to conclude the proof, it suffices to find a component $E_a$ of $D$ such that 
$\b$ is epimorphic.
Below we explain that such component exists.

Indeed, by assumption we have
$$
1\le \chi(\O_D) =-\frac12D(D+K_X)=-\frac12\sum_ik_iE_i(D+K_X)=\frac12\sum_ik_i\left(2+r_i-k_ir_i-\sum_{j\colon E_j\cap E_i\ne \emptyset, j\ne i}k_j\right)
$$
Therefore for some $a$ one has 
$2+r_a-k_ar_a-\sum_{j}k_j$
positive. Consequently,
$2-r_a(k_a-1)-\sum_{j}k_j\ge 1$
and 
$$\sum_{j\colon E_j\cap E_a\ne \emptyset, j\ne a}k_j\le 1-r_a(k_a-1).$$

We claim that $E_a$ is a suitable component. Indeed,  the map $\b$ has the form
$$H^0(E_a,\O_{E_a}(-r_a(k_a-1)))\to \oplus_j H^0(E_a, \O_{E_a}(-r_a(k_a-1))|_{k_jP_j})$$
where $P_j=E_j\cap E_a$. Choose an affine coordinate $x$ on $E_a\cong\P^1$ such that 
none of the points $P_j$ is $\infty$. Then $H^0(E_a,\O_{E_a}(-r_a(k_a-1)))$ is isomorphic to a subspace in $\k[x]$ of  polynomials of degree $\le -r_a(k_a-1)$. The  map 
$$H^0(E_a,\O_{E_a}(-r_a(k_a-1)))\to H^0(E_a, \O_{E_a}(-r_a(k_a-1))|_{k_jP_j})$$
is the evaluation of a polynomial and its derivatives of order $< k_j$ in the point $P_j$.
By Hermite interpolation, the map $\b$ is epimorphic if and only if 
%$$\dim(H^0(E_a,\O_{E_a}(-r_a(k_a-1))))\ge \dim(\oplus_j H^0(E_a, %\O_{E_a}(-r_a(k_a-1))|_{k_jP_j})).$$
%The latter inequality holds because 
$1-r_a(k_a-1) \ge \sum_{j}k_j$, what is true.
Thus the lemma is proved.
\end{proof}

\begin{corollary}
\label{corollary_positive}
Let $D$ be a left-orthogonal divisor on  a surface $X$ with $h^1(X,\O_X)=h^2(X,\O_X)=0$. Let $E\subset D$ be a component with multiplicity $k$ and self-intersection $r$.  
Then one of the following holds:
\begin{enumerate}
\item $r\le 0, k$ any,
\item $r=1,k=1$ or $2$,
\item $r>1,k=1$.
\end{enumerate}
\end{corollary}

\section{Criterion for strong left-orthogonality}

\begin{theorem}
\label{th_slo2}
An effective divisor $D$ on a surface $X$ with $h^1(X,\O_X)=h^2(X,\O_X)=0$ is strongly left-orthogonal if and only if the following conditions hold
\begin{enumerate}
\item $D$ is a tree of projective lines.
\item
$p_a(D)=0$ and for any connected divisor $D'$ such that $0<D'\le D$ one has $p_a(D')\le 0$.
\item
For any connected divisor $D'$ such that $0<D'\le D$ one has $p_a(D')\le 1+D\cdot D'$.
\end{enumerate}
\end{theorem}

\begin{proof}
By Theorem~\ref{th_slo}, left-orthogonality of $D$ is equivalent to conditions (1) and (2). By Lemma~\ref{lemma_lo}.(2), we have to demonstrate (assuming left-orthogonality) that condition (3) is equivalent to $h^1(\O_D(D))=0$.

Let $D'\le D$ be an effective connected divisor. One has an exact sequence of sheaves 
$$0\to \II_{D',D}(D)\to \O_{D}(D)\to \O_{D'}(D)\to 0.$$
It gives an exact sequence 
$$0=H^1(X,\O_{D}(D))\to H^1(X,\O_{D'}(D))\to H^2(X,\II_{D',D}(D))=0$$
therefore $h^1(D',\O_{D'}(D))=0$ and $\chi(\O_{D'}(D))\ge 0$.
Also consider exact sequence 
$$0\to \O_X(D-D')\to \O_X(D)\to \O_{D'}(D)\to 0.$$
One has 
\begin{multline*}
0\le \chi(\O_{D'}(D))=\chi(\O_X(D))-\chi(\O_X(D-D'))=\\
=\frac12D(D-K_X)-\frac12(D-D')(D-D'-K_X)=DD'-\frac12(-D')(-D'-K_X)=DD'+\chi(\O_{D'}).
\end{multline*}
Hence, $p_a(\O_{D'})\le 1+DD'$.

Implication (3) $\Rightarrow$ ($h^1(\O_D(D))=0$) is proved in the next lemma.
\end{proof}

\begin{lemma}
\label{lemma_cond4}
In notation of Theorem~\ref{th_slo2}, suppose $D$ is left-orthogonal and for any 
connected divisor $D'$ such that $0<D'\le D$ one has $\chi(\O_{D'})+ DD'\ge 0$. Then $h^1(\O_{D}(D))=0$.
\end{lemma}

\begin{proof}
The proof is similar to the proof of Lemma~\ref{lemma_main1}.
We will prove by induction in $D'$ that for any connected divisor $D'$ such that $0<D'\le D$ one has $h^1(\O_{D'}(D))=0$.
Denote $b_i=D\cdot E_i$. By assumptions of lemma, one has $\chi(\O_{E_i})+D\cdot E_i\ge 0$, hence $b_i\ge -\chi(\O_{E_i})=-1$. 

As a base of induction we consider the case $D'=D_i=k_iE_i$. For brevity we drop off index $i$ below. We have $\O_{D'}(D)\cong \O_{kE}(b)$ and we claim that vanishing 
$h^1(\O_{kE}(b))=0$ follows from Lemma~\ref{lemma_cond3}. Indeed, for $r\le 0$ the first condition of Lemma~\ref{lemma_cond3} holds. And for $r>0$ we have $b=D\cdot E>0$ so by Corollary~\ref{corollary_positive} conditions 2 or 3 hold.

For the induction step, consider a connected  divisor $D'=\sum k'_iE_i$ such that $0<D'\le D$.
Let $E_a$ be a component of $D'$, let $D''=D'-E_a$.  Denote by $\II$ the sheaf of ideals of the subscheme $E_a\subset X$. Then the sheaf of ideals of subscheme $D''_a\subset D'_a$ is isomorphic to $\II^{k'_a-1}/\II^{k'_a}\cong (\II/\II^2)^{k'_a-1}\cong \O_{E_a}(-r_a(k'_a-1))$. Consider a commutative diagram
\begin{equation}
\tiny
\label{eq_hataab}
\xymatrix{
0\ar[r] & H^0(X,\II^{k'_a-1}/\II^{k'_a}(D)) \ar[r] \ar[d]^{\hat\beta} & 
\oplus H^0(X,\O_{D'_i}(D)) \ar[r] \ar[d]^{\hat\a_{D'}} & 
\oplus H^0(X,\O_{D''_i}(D)) \ar[r]\ar[d]^{\hat\a_{D''}} & 0\\
0\ar[r] & \bigoplus\limits_{j: E_j\cap E_a\ne \emptyset, j\ne a}  H^0\left(\left(\II^{k'_a-1}/\II^{k'_a}(D)\right)|_{D'_j}\right) \ar[r] &
\oplus H^0(X,\O_{D'_i\cap D'_j}(D)) \ar[r] & 
\oplus H^0(X,\O_{D''_i\cap D''_j}(D)) \ar[r] & 0. 
}
\end{equation}
Consider exact sequence of sheaves
\begin{equation*}
%\label{eq_118}
0\to \II^{k'_a-1}/\II^{k'_a}(D)\to \O_{k'_aE_a}(D)\to \O_{(k'_a-1)E_a}(D)\to 0.
\end{equation*}
In its long sequence of cohomology one has $H^1(X,\II^{k'_a-1}/\II^{k'_a}(D))=0$. Indeed,
$\II^{k'_a-1}/\II^{k'_a}(D)\cong \O_{E_a}(b_a-r_a(k'_a-1))$. We claim that $b_a-r_a(k'_a-1)\ge -1$, but this is essentially proven in the induction base.
Therefore the first row of (\ref{eq_hataab}) is exact. 
The second row of (\ref{eq_hataab}) is also exact, see the proof of Lemma~\ref{lemma_main1}. 

Thus, diagram (\ref{eq_hataab}) gives an exact sequence
$$\coker \hat\b\to \coker \hat\a_{D'}\to \coker \hat\a_{D''}.$$
We claim that $\hat\a_{D''}$ is epimorphic. That can be deduced applying the induction hypothesis to connected components of $D''$, see the proof of Lemma~\ref{lemma_main1}.
Consequently, if $\hat\b$ is epimorphic then $\hat\a_{D'}$ also is epimorphic. Hence, to conclude the proof, it suffices to find a component $E_a$ of $D'$ such that 
$\hat\b$ is epimorphic.
Below we explain that such component exists.

One has $\II^{k'_a-1}/\II^{k'_a}(D)\cong\O_{E_a}(b_a-r_a(k'_a-1))$. Hence the map $\hat\b$ is the restriction map
$$H^0(E_a,\O_{E_a}(b_a-r_a(k'_a-1)))\to \oplus_j H^0(E_a, \O_{E_a}(b_a-r_a(k'_a-1))|_{k'_jP_j})$$
where $P_j=E_j\cap E_a$. By Hermite interpolation (see the proof of Lemma~\ref{lemma_main1}),  the map~$\hat\b$ is epimorphic if and only if 
\begin{equation}
\label{eq_hermite}
1+b_a-r_a(k'_a-1) \ge \sum_{j\colon E_j\cap E_a\ne\emptyset, j\ne a}k'_j.
\end{equation}
Let us check that (\ref{eq_hermite}) holds for some $a$.

By assumption we have $\chi(\O_{D'})+DD'\ge 0$ and by left-orthogonality of $D$ (see Theorem~\ref{th_slo}) we have $\chi(\O_{D'})\ge 1$. Therefore 
\begin{multline*}
\frac12\le \chi(\O_{D'})+\frac12 D D' =\frac12D'(-D'-K_X+D)=
\frac12\sum_ik'_iE_i(D-K_X-D')=\\
=\frac12\sum_ik'_i\left(b_i+2+r_i-k'_ir_i -\sum_{j\colon E_j\cap E_i\ne \emptyset, j\ne i}k'_j\right).
\end{multline*} 

Therefore at least one summand on the right is positive. One has 
$$1\le b_a+2+r_a-k'_ar_a -\sum_{j\colon E_j\cap E_a\ne \emptyset, j\ne a}k'_j$$ 
for some $a$. Hence (\ref{eq_hermite}) holds.
\end{proof}

\section{Examples and remarks}

Let $D=\sum_ik_iE_i$ be a tree of projective lines and $r_i=E_i^2$.
Then $p_a(D)$ can be calculated explicitly: 
\begin{multline}
\label{eq_genus}
p_a(D)=1+\frac12D(D+K_X)=1+\frac12\left(\sum_ik_i^2r_i+2\sum_{i<j\colon E_i\cap E_j\ne \emptyset}k_ik_j-\sum_ik_i(2+r_i)\right)=\\
=1+\sum_i\left(r_i\frac{k_i(k_i-1)}2-k_i\right)-\sum_{i<j\colon E_i\cap E_j\ne \emptyset}k_ik_j.
\end{multline}

\begin{corollary}
\label{corollary_magicformula}
Let $D=\sum_i k_iE_i\ge 0$ be a left-orthogonal effective divisor on a surface~$X$ with $h^1(X,\O_X)=h^2(X,\O_X)=0$. Then one has 
\begin{equation}
\label{eq_magicformula}
1+\frac12D^2=\sum_{i}\left(k_i\left(1+\frac12  r_i\right)\right).
\end{equation}
\end{corollary}
\begin{proof}
Above equation is equivalent to $p_a(D)=0$.
\end{proof}

\begin{predl}
\label{prop_tree0}
Let $D$ be a reduced tree of projective lines. Then $p_a(D)=0$.
\end{predl}
\begin{proof}
It follows from formula~(\ref{eq_genus}).
\end{proof}

%Below we present important corollaries of criteria \ref{th_slo} and \ref{th_slo2}.
\begin{predl}
\label{prop_slored}
Let $D=\sum_i E_i$ be a reduced tree of projective lines on a surface $X$ with $h^1(X,\O_X)=h^2(X,\O_X)=0$. Then 
\begin{enumerate}
\item $D$ is left-orthogonal.
\item $D$ is strongly left-orthogonal if and only if $D\cdot D'\ge -1$ for any connected divisor $D'$ such that $0<D'\le D$.
\end{enumerate}
\end{predl}
\begin{proof}

\begin{enumerate}
\item We use Theorem~\ref{th_slo}. Clearly, condition (1) holds. We have  $p_a(D')=0$ for any connected subdivisor in $D$, see Proposition~\ref{prop_tree0}. Therefore, condition (2) also holds and $D$ is left-orthogonal.

\item We use Theorem~\ref{th_slo2}. For a connected subdivisor $D'$ of $D$, one has $p_a(D')=0$. Thus condition (3) has the form $D\cdot D'\ge -1$.
\end{enumerate}
\end{proof}

\begin{predl}
Let $X$ be a surface with $h^1(X,\O_X)=h^2(X,\O_X)=0$. Let $D=kE$ be a projective line with multiplicity $k$ and $E^2=r$. Then 
\begin{enumerate}
\item $D$ is left-orthogonal if and only if one of the following conditions is satisfied:
\begin{enumerate}
\item $k=1$, 
\item $k=2, r=1$.
\end{enumerate}

\item $D$ is strongly left-orthogonal if and only if one of the following conditions is satisfied:
\begin{enumerate}
\item $k=1, r\ge -1$, 
\item $k=2, r=1$.
\end{enumerate}
\end{enumerate}
\end{predl}

\begin{example}
Let $D$ be chain of $n$ smooth rational curves with self-intersections $-2$. Such $D$ is an exceptional divisor of the minimal resolution of Du Val singular point of type $A_{n}$.
Then $D$ is left-orthogonal but not strongly left-orthogonal. The only subdivisor $D'$ of $D$ such that condition $DD'\ge -1$ of Proposition~\ref{prop_slored} fails is $D'=D$. One has $D^2=-2$.
\end{example}

\begin{example}
Let $D=\sum_{i=1}^5E_i$ be a chain of smooth rational curves with self-intersections $0,-3,-2,-3,0$. Then $D$ is left-orthogonal but not strongly left-orthogonal. The only subdivisor $D'$ of $D$ such that condition $DD'\ge -1$ of Proposition~\ref{prop_slored} fails is $D'=E_2+E_3+E_4$. One has $DD'=-2$.
\end{example}

\bigskip
\begin{remark}
Let $D=\sum_i k_iE_i$ be a tree of projective lines with self-intersections $r_i$. Note that condition (2) of left-orthogonality from Theorem~\ref{th_slo} due to the formula (\ref{eq_genus}) can be expressed as a system of equalities and inequalities involving values $k_i$ and $r_i$. But not any collection of numbers satisfying the above system is realized by some left-orthogonal divisor on a surface. For example, $D=E_1+E_2$ with $E_1\cdot E_2=1$, $E_1^2=E_2^2=2$ cannot be a divisor on a surface due to Hodge index theorem. 
\end{remark}

In fact, Hodge index theorem imposes rather strong conditions on components with positive self-intersection of a left-orthogonal divisor which we present below. 

\begin{predl}
\label{prop_twopositive}
Let $D$ be a left-orthogonal effective divisor on a surface $X$ with $h^1(X,\O_X)=h^2(X,\O_X)=0$. Let $E_1$ and $E_2$ be two prime components of $D$ such that $r_1,r_2>0$. Then $D=E_1+E_2$ and $r_1=r_2=1$.

Therefore, $D$ can contain at most two components with positive self-intersection. Moreover, $D$ contains two positive components only in the case $D=E_1+E_2, E_i^2=1, D^2=4$.
\end{predl}
\begin{proof}
The Gram matrix $G$ of the intersection form on vectors $E_1,E_2\in NS(X)$ is either 
$$\begin{pmatrix} r_1& 1\\
1& r_2
\end{pmatrix} \quad\text{or}\quad \begin{pmatrix} r_1& 0\\
0& r_2
\end{pmatrix}.$$
It follows from Hodge index theorem that $G$ cannot be positive definite.
Hence, the second case is impossible. The first case is possible only if $\det G=r_1r_2-1\le 0$. Therefore $r_1r_2=1$ and $r_1=r_2=1$.

Suppose $D$ has more than two irreducible components. Since components of $D$ form a tree, we can choose a component $E_3$ such that $E_1\cdot E_3=1, E_2\cdot E_3=0$. Then 
the Gram matrix of the intersection form on vectors $E_2,E_1+cE_3\in NS(X)\otimes \Q$ is
$$
\begin{pmatrix} 1& 1 \\
1& 1+2c+c^2r_3
\end{pmatrix}.
$$
Its corner minors are $1$ and $c(2+cr_3)$. If one takes $c>0$ and small enough, this matrix is positive definite by Sylvester's criterion. That gives a contradiction to Hodge index theorem.

Therefore $D=k_1E_1+k_2E_2$. By Corollary~\ref{corollary_magicformula} we have
$$1+\frac12(k_1+k_2)^2=2+\frac12(k_1^2+k_2^2).$$
It follows that $k_1k_2=1$, hence $k_1=k_2=1$.
\end{proof}

\begin{corollary}
Let $D$ be an effective left-orthogonal divisor on a surface $X$ with $h^1(X,\O_X)=h^2(X,\O_X)=0$. Suppose $E$ is a prime component of $D$ with multiplicity $>1$ and $r=E^2>0$. Then $D=2E$, $E^2=1$.
\end{corollary}
\begin{proof}
We claim that $E$ is linearly equivalent to some other prime divisor. Indeed, $O_E$ is a quotient-sheaf of $\O_D$, hence $h^1(X,\O_E)=0$. Clearly, $h^0(X,\O_E)=1$ therefore by Lemma~\ref{lemma_lo}.(1) divisor $E$ is left-orthogonal. Also, $h^1(X,\O_E(E))\cong h^1(E,\O_E(r))=0$. Thus by Lemma~\ref{lemma_lo}.(2) divisor $E$ is strongly left-orthogonal.

By Riemann-Roch formula we have
\begin{align*}
\chi(X,E)&=\chi(\O_X)+\frac{E(E-K_X)}2=1+\frac{E\cdot E-E\cdot K_X}2;\\
0=\chi(X,-E)&=\chi(\O_X)+\frac{-E(-E-K_X)}2=1+\frac{E\cdot E+E\cdot K_X}2;
\end{align*}
hence $h^0(X,E)=\chi(X,E)=2+E^2\ge 2$. It follows that there exists a prime divisor  $E'\ne E$, linearly` equivalent to $E$. Then $D-E+E'$ is a left-orthogonal divisor with two components $E,E'$ such that $E^2,E'^2>0$. By Proposition~\ref{prop_twopositive}, one has $D-E+E'=E+E'$ and $E^2=1$. Consequently $D=2E$.
\end{proof}

\end{document}